\nonstopmode
\documentclass[10pt,twocolumn]{quick-article}

\usepackage{smallcolumns}

\makeatletter
\newcounter{tmp@counter@storage}
\newcommand{\tmpsetcounter}[2]{%
  \setcounter{tmp@counter@storage}{\value{#1}}%
  \setcounter{#1}{#2}\addtocounter{#1}{-1}}
\newcommand{\tmprestorecounter}[1]{%
  \setcounter{#1}{\value{tmp@counter@storage}}}
\makeatother

\usepackage{hyperref}

\usepackage{color}
\usepackage{graphics}
\usepackage{url}

\usepackage{aoa}
\usepackage{mathtools}
\def\etal{\textit{et~al.}}
\def\email#1{\texttt{\href{mailto:#1}{#1}}}

\newcommand{\EIS}[1]{{\em EIS\/}~{\bf #1}}

\usepackage[colorinlistoftodos,textsize=small]{todonotes}

\def\mNode{\circ}      
\def\mLeaf{\bullet}    
\def\mINode{\diamond}  

\def\mEdge{\mNode-\mNode}
\def\mDEdge{\mNode\rightarrow\mNode}

\def\mLEdge{\mLeaf-\mINode}
\def\mLDEdge{\mLeaf \rightarrow \mINode}
\def\mDLEdge{\mINode \rightarrow \mLeaf}
\def\mDIEdge{\mINode \rightarrow \mINode}
\def\mIEdge{\mINode \rightarrow \mINode}

\def\DHtext{DH}
\def\clsDH{\cls{\DHtext}}
\def\clsDHrl{\cls[\mLeaf]{\DHtext}}

\def\TLPtext{\text{3}LP}
\def\clsTLP{\cls{\TLPtext}}
\def\clsTLPrl{\cls[\mLeaf]{\TLPtext}}

\begin{document}


\title{An Exact Enumeration of Distance-Hereditary Graphs}

\author{%
  C\'{e}dric Chauve\thanks{Dept.~of~Mathematics, Simon Fraser University, 8888
    University Drive, V5A~1S6, Burnaby (BC), Canada,
    \email{cedric.chauve@sfu.ca}}\and %
  \'Eric Fusy\thanks{CNRS \& LIX, \'Ecole Polytechnique, 91120 Palaiseau,
    France, \email{fusy@lix.polytechnique.fr}}\and %
  J\'{e}r\'{e}mie Lumbroso\thanks{Dept.~of~Computer~Science, Princeton
    University, 35~Olden~Street, Princeton, NJ 08540, USA,
    \email{lumbroso@cs.princeton.edu}}}

\date{}

\maketitle

\begin{abstract}
  Distance-hereditary graphs form an important class of graphs, from the
  theoretical point of view, due to the fact that they are the totally
  decomposable graphs for the split-decomposition. The previous best
  enumerative result for these graphs is from Nakano~\etal~(J. Comp. Sci.
  Tech., 2007), who have proven that the number of distance-hereditary
  graphs on $n$ vertices is bounded by $\cramped{2^{\lceil 3.59n\rceil}}$.
  
  In this paper, using classical tools of enumerative combinatorics, we
  improve on this result by providing an \emph{exact} enumeration of
  distance-hereditary graphs, which allows to show that the number of
  distance-hereditary graphs on $n$ vertices is tightly bounded by
  $\cramped{(7.24975\ldots)^n}$---opening the perspective such graphs could be
  encoded on $3n$ bits. We also provide the exact enumeration and
  asymptotics of an important subclass, the 3-leaf power graphs.

  Our work illustrates the power of revisiting graph decomposition results
  through the framework of analytic combinatorics.
\end{abstract}


\section*{Introduction}

The decomposition of graphs into tree-structures is a fundamental paradigm
in graph theory, with algorithmic and theoretical
applications~\cite{Buixuan08}. In the present work, we are interested in
the \emph{split-decomposition}, introduced by Cunningham and
Edmonds~\cite{Cunningham82,CuEd80} and recently revisited by
Gioan~\etal~\cite{GiPa12,GiPaTeCo13, ChMoRa12}. For the classical modular
and split-decomposition, the \emph{decomposition tree} of a graph $G$ is a
tree (rooted for the modular decomposition and unrooted for the split
decomposition) of which the leaves are in bijection with the vertices of
$G$ and whose internal nodes are labeled by indecomposable (for the chosen
decomposition) graphs; such trees are called \emph{graph-labeled trees} by
Gioan and Paul~\cite{GiPa12}. Moreover, there is a one-to-one
correspondence between such trees and graphs. The notion of a graph being
\emph{totally decomposable} for a decomposition scheme translates into
restrictions on the labels that can appear on the internal nodes of its
decomposition tree. For example, for the split-decomposition, totally
decomposable graphs are the graphs whose decomposition tree's internal
nodes are labeled only by cliques and stars; such graphs are called
\emph{distance-hereditary graphs}. They generalize the well-known
\emph{cographs}, the graphs that are totally decomposable for the modular
decomposition, and whose enumeration has been well studied, in particular
by Ravelomanana and Thimonier~\cite{RaTh01}, also using techniques from
analytic combinatorics

Efficiently encoding graph classes\footnote{By which we mean, describing
  any graph from a class with as few bits as possible, as described for
  instance by Spinrad~\cite{Spinrad03}.} is naturally linked to the
enumeration of such graph classes. Indeed the number of graphs of a given
class on $n$ vertices implies a lower bound on the best possible encoding
one can hope for. Until recently, few enumerative properties were known
for distance-hereditary graphs, unlike their counterpart for the modular
decomposition, the cographs. The best result so far, by
Nakano~\etal~\cite{NaUeUn09}, relies on a relatively complex encoding on
$4n$ bits, whose detailed analysis shows that there are at most
$\cramped{2^{\lfloor 3.59n\rfloor}}$ unlabeled distance-hereditary graphs
on $n$ vertices. However, using the same techniques, their result also
implies an upper-bound of $2^{3n}$ for the number of unlabeled cographs
on $n$ vertices, which is far from being optimal for these graphs, as it
is known that, asymptotically, there are
$C \cramped{d^n}/\cramped{n^{3/2}}$ such graphs where $C = 0.4126\dots$
and $d = 3.5608\dots$~\cite{RaTh01}. This suggests there is room for
improving the best upper bound on the number of distance-hereditary graphs
provided by Nakano~\etal~\cite{NaUeUn09}, which was the main purpose of
our present work.

\subsection*{This paper.}

Following a now well established approach, which enumerates graph classes
through a tree representation, when available (see for example the survey
by Gim\'enez and Noy~\cite{GiNo09} on tree-decompositions to count
families of planar graphs), we provide \emph{combinatorial
  specifications}, in the sense of Flajolet and Sedgewick~\cite{FlSe09},
of the split-decomposition trees of distance-hereditary graphs and 3-leaf
power graphs, both in the labeled and unlabeled cases. From these
specifications, we can provide \emph{exact enumerations},
\emph{asymptotics}, and leave open the possibility of uniform random
samplers allowing for further empirical studies of statistics on these
graphs (see Iriza~\cite{Iriza15}).

In particular, we show that the number of distance-hereditary graphs on
$n$ vertices is bounded from above by $\cramped{2^{3n}}$, which naturally
opens the question of encoding such graphs on $3n$ bits, instead of $4n$
bits as done by Nakano~\etal~\cite{NaUeUn09}. We also provide similar
results for 3-leaf power graphs, an interesting class of distance
hereditary graphs, showing that the number of 3-leaf power graphs on $n$
vertices is bounded from above by $\cramped{2^{2n}}$.

\subsection*{Main results.}

Our main contribution is to introduce the idea of symbolically specifying
the trees arising from the split-decomposition, so as to provide the
(previously unknown) exact enumeration of certain important classes of
graphs.

Our grammars for distance-hereditary graphs are in
Subsection~\ref{sec:dh}, and our grammars for 3-leaf power graphs are in
Subsection~\ref{sec:3lp}. We provide here the corollary that gives the
beginning of the exact enumerations for the unlabeled and unrooted
versions of both classes\footnote{With the symbolic grammars, it is then
  easy to establish recurrences~\cite{FlZiVa94, Zimmermann91} to
  efficiently compute the enumeration--to the extent that we were
  trivially able to obtain the first 10\,000 terms of the enumerations.
  See a survey by Flajolet and Salvy~\cite[\S 1.3]{FlSa95} for more
  detail.}.

\begin{corollary}[Enumeration of connected, unlabeled, unrooted
  distance-hereditary graphs]
  \label{corr:enum-unlabeled-unrooted-dh}%
  The first few terms of the enumeration, \EIS{A00000}, are
  \begin{align*}
      &1, 1, 2, 6, 18, 73, 308, 1484, 7492, 40010, 220676,\\
      &\qquad 1253940, 7282316, 43096792, 259019070,\\
      &\qquad 1577653196, 9720170360, 60492629435\ldots
  \end{align*}
  and the asymptotics is
  $c\cdot\cramped{7.249751250\ldots^n}\cdot \cramped{n^{-5/2}}$ with
  $c\approx 0.02337516194\ldots$.
\end{corollary}

\begin{corollary}[Enumeration of connected, unlabeled, unrooted $3$-leaf
  power graphs]
  \label{corr:enum-unlabeled-unrooted-tlp}%
  The first few terms of the enumeration, \EIS{A00000}, are
  \begin{align*}
      &1, 1, 2, 5, 12, 32, 82, 227, 629, 1840, 5456, 16701,\\
      &\qquad 51939, 164688, 529070, 1722271, 5664786,\\
      &\qquad 18813360, 62996841, 212533216\ldots
  \end{align*}
  and the asymptotics is
  $c\cdot\cramped{3.848442876\ldots^n}\cdot \cramped{n^{-5/2}}$ with
  $c\approx 0.70955825396\ldots$.
\end{corollary}


\section{Definitions and Preliminaries}

For a graph $G$, we denote by $V(G)$ its vertex set and $E(G)$ its edge
set. Moreover, for a vertex $x$ of a graph $G$, we denote by $N(x)$ the
neighbourhood of $x$, that is the set of vertices $y$ such that
$\{x,y\}\in E(G)$; this notion extends naturally to vertex sets: if
$\cramped{V_1}\subseteq V(G)$, then $N(\cramped{V_1})$ is the set of
vertices defined by the (non-disjoint) union of the neighbourhoods of the
vertices in $\cramped{V_1}$. Finally, the subgraph of $G$ induced by a
subset $\cramped{V_1}$ of vertices is denoted by $G[\cramped{V_1}]$.

A graph on $n$ vertices is \emph{labeled} if its vertices are identified
with the set $\{1,\dots,$ $n\}$, with no two vertices having the same
label. A graph is \emph{unlabeled} if its vertices are indistinguishable.

A clique on $k$ vertices, denoted $\cramped{K_k}$ is the complete graph on
$k$ vertices (\textit{i.e.}, there exists an edge between every pair of
vertices). A star on $k$ vertices, denoted $\cramped{S_k}$, is the graph
with one vertex of degree $k-1$ (the \emph{center} of the star) and $k-1$
vertices of degree $1$ (the \emph{extremities} of the star).

\subsection{Split-decomposition trees.\label{subsec:split}}

We first introduce the notion of \emph{graph-labeled tree}, due to Gioan
and Paul~\cite{GiPa12}, then define the split-decomposition and the
corresponding tree, described as a graph-labeled tree.

\begin{definition}\label{def:glt}
  A graph-labeled tree $(T,\mathcal{F})$ is a tree\footnote{This is a
    non-plane tree: the ordering of the children of an internal node does
    not matter---this is why in most of our grammars we describe the
    children as a $\Set$ instead of a $\Seq$, a sequence.} $T$ in which
  every internal node $v$ of degree $k$ is labeled by a graph
  $\cramped{G_v} \in \mathcal{F} $ on $k$ vertices, such that there is a
  bijection $\cramped{\rho_v}$ from the edges of $T$ incident to $v$ to
  the vertices of $\cramped{G_v}$.
\end{definition}


\begin{definition}
  A \emph{split}~\cite{Cunningham82} of a graph $G$ with vertex set $V$ is
  a bipartition $(\cramped{V_1},\cramped{V_2})$ of $V$ (\textit{i.e.},
  $V=\cramped{V_1}\cup V_2$, $\cramped{V_1}\cap \cramped{V_2}=\emptyset$)
  such that
  \begin{enumerate}[label=(\alph*), noitemsep, nosep]
  \item $|V_1|>2$ and $|V_2|>2$;
  \item every vertex of $N(V_1)$ is adjacent to every of $N(V_2)$.
  \end{enumerate}
\end{definition}

\noindent A graph without any split is called a \emph{prime} graph. A
graph is \emph{degenerate} if any partition of its vertices without a
singleton part is a split: cliques and stars are the only such 
graphs.

\begin{figure*}[t!]
  \centering
  \begin{minipage}[b]{.45\linewidth}
    \centering
    \includegraphics[scale=0.38]{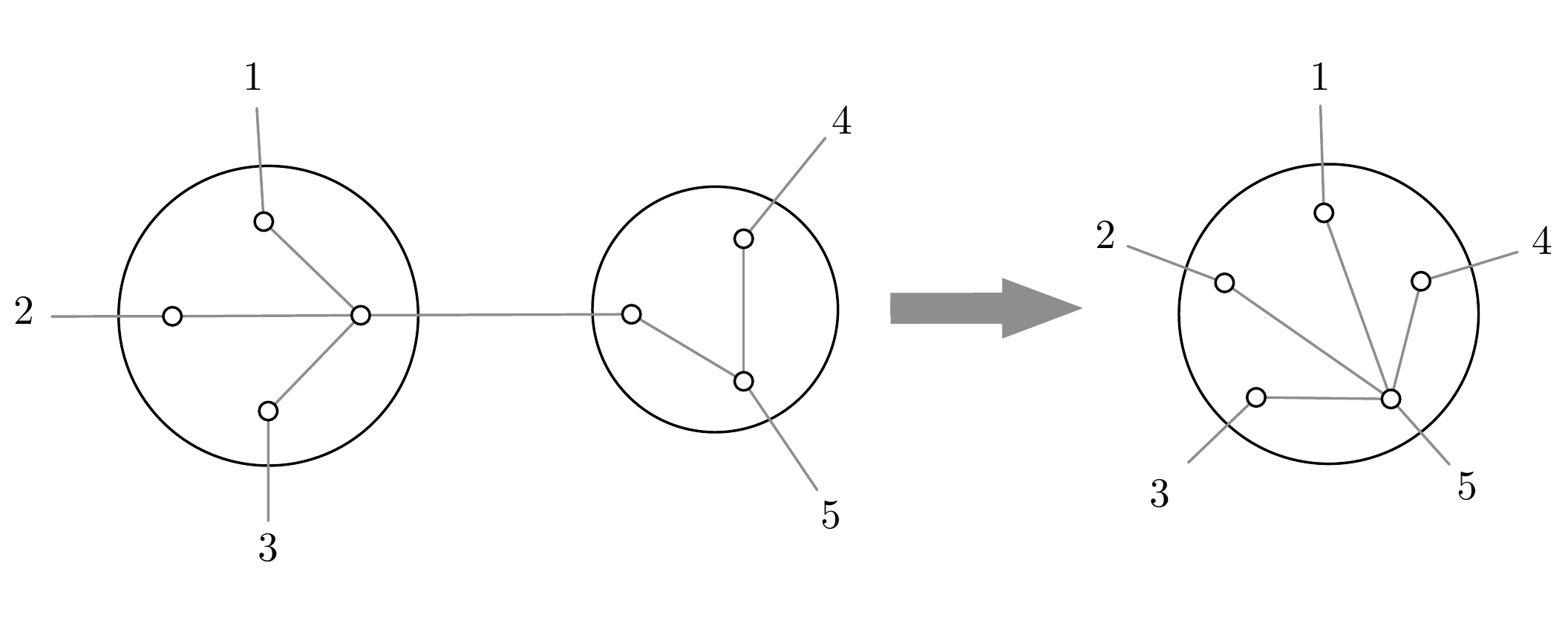}
    \subcaption{Example of a star-join.
      \label{fig:star-join}}
  \end{minipage}\hspace{0.08\linewidth}%
  \begin{minipage}[b]{.45\linewidth}
    \centering
    \includegraphics[scale=0.38]{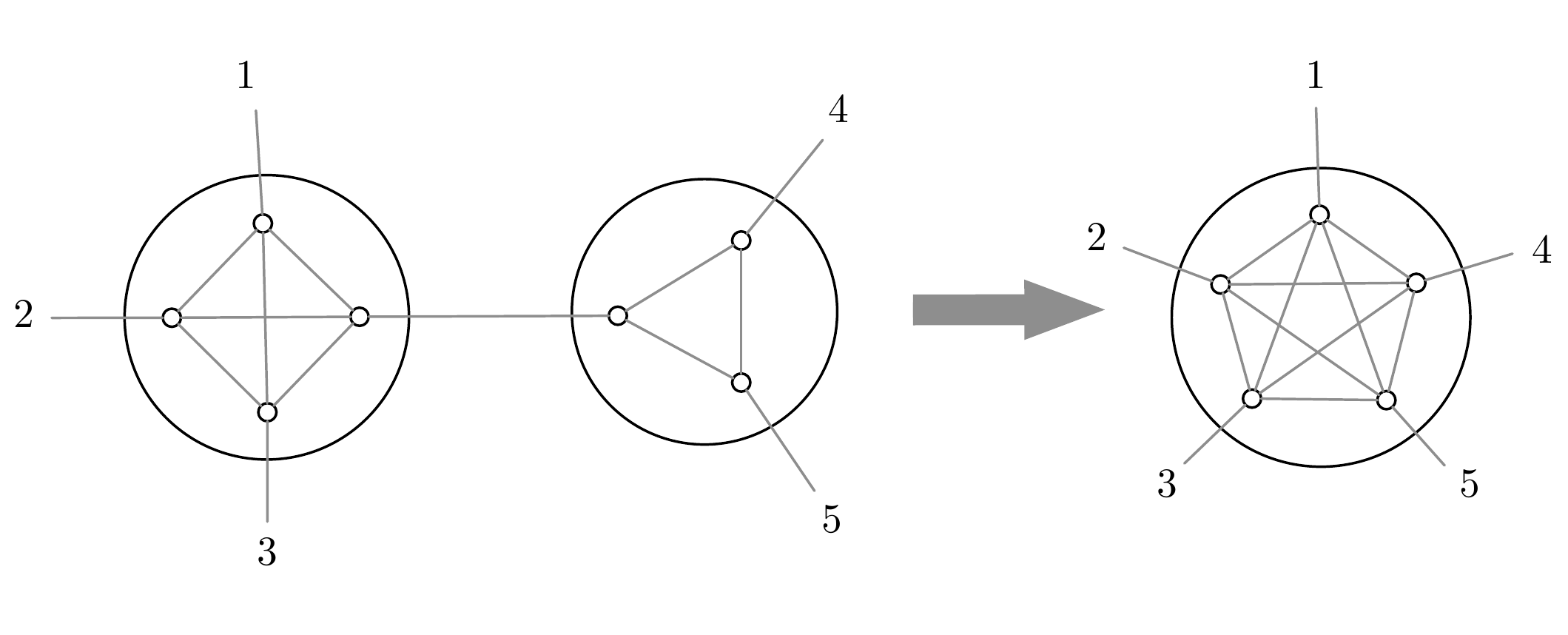}
    \subcaption{Example of a clique-join.\label{fig:clique-join}}
  \end{minipage}
  \caption{The star-join and clique-join operations result in the merging
    of two internal nodes of a split-decomposition tree. A
    split-decomposition tree in which neither one of these operations may
    be applied (and in which all non-clique and non-star nodes are prime
    nodes) is said to be \emph{reduced}.\label{fig:reduced}}
\end{figure*}

Informally, the split-decomposition of a graph $G$ consists in finding a
split $(\cramped{V_1}, \cramped{V_2})$ in $G$, followed by decomposing $G$
into two graphs $\cramped{G_1}=G[\cramped{V_1}\cup \{\cramped{x_1}\}]$
where $\cramped{x_1}\in N(\cramped{V_1})$ and
$\cramped{G_2}=G[\cramped{V_2}\cup \{\cramped{x_2}\}]$ where
$\cramped{x_1}\in N(\cramped{V_2})$ and then recursively decomposing
$\cramped{G_1}$ and $\cramped{G_2}$. This decomposition naturally defines
an unrooted tree structure of which the internal vertices are labeled by
degenerate or prime graphs and whose leaves are in bijection with the
vertices of $G$, called a \emph{split-decomposition tree}. A
split-decomposition tree $(T,\mathcal{F})$ with $\mathcal{F}$ containing
only cliques with at least three vertices and stars with at least three
vertices is called a \emph{clique-star tree}.

It can be shown that the split-decomposition tree of a graph might not be
unique (\textit{i.e.}, that several decompositions sequences of a given
graph can lead to different split-decomposition trees), but following
Cunningham~\cite{Cunningham82}, we obtain the following uniqueness result,
reformulated in terms of graph-labeled trees by Gioan and
Paul~\cite{GiPa12}.

\begin{theorem*}[Cunningham~\cite{Cunningham82}]
  For every connected graph $G$, there exists a unique split-decomposition
  tree such that:
  \begin{enumerate}[label=(\alph*), noitemsep, nosep]
  \item every non-leaf node has degree at least three;
  \item no tree edge links two vertices with clique labels;
  \item no tree edge links the center of a star to the extremity of
    another star.
  \end{enumerate}
\end{theorem*}

\noindent Such a tree is called \emph{reduced}, and this theorem
establishes a one-to-one correspondence between graphs and their reduced
split-decomposition trees. So enumerating the split-decomposition trees of
a graph class provides an enumeration for the corresponding graph class,
and we rely on this property in the following sections.

\subsection{Decomposable structures.\label{subsec:FS}}

In order to enumerate classes of split-decomposition trees, we use the
framework of decomposable structures, described by Flajolet and
Sedgewick~\cite{FlSe09}. We refer the reader to this book for details and
outline below the basics idea.

We denote by $\clsAtom$ the combinatorial family composed of a single
object of size $1$, usually called \emph{atom} (in our case, these refer
to a leaf of a split-decomposition tree, \textit{i.e.}, a vertex of the
corresponding graph).

Given two disjoint families $\cls{A}$ and $\cls{B}$ of combinatorial
objects, we denote by $\cls{A} + \cls{B}$ the \emph{disjoint union} of the
two families and by $\cls{A} \times \cls{B}$ the \emph{Cartesian product}
of the two families.

Finally, we denote by $\Set{\cls{A}}$ (resp.
$\cramped{\Set[\geqslant k]{\cls{A}}}$, $\cramped{\Set[k]{\cls{A}}}$) the
family defined as all sets (resp. sets of size at least $k$, sets of size
exactly $k$) of objects from ${\cls{A}}$, and by
$\cramped{\Seq[\geqslant k]{\cls{A}}}$, the family defined as all
sequences of at least $k$ objects from ${\cls{A}}$.

\begin{remark}
  Because this paper deals with classes both rooted (either at a
  vertex/leaf or an internal node) and unrooted, we use some notations to
  keep these distinct. But these notations are purely for clarity. 

  For instance, while we use $\cramped{\clsAtom_{\bullet}}$ to denote a
  \emph{rooted} vertex, and $\clsAtom$ to denote an \emph{unrooted}
  vertex, these are both translated in the same way in the associated
  generating functions and enumerations.
\end{remark}

\begin{remark}
  Decomposable structures specified by these grammars can either be:
  \begin{itemize}
  \item \emph{labeled}: in a given object, each atom is labeled by a
    distinct number between 1 and $n$ (the size of the object); this means
    that each ``skeleton'' of an object appears in $n!$ copies, for each
    of the possible way of labeling its individual atoms, and because each
    atom is distinguished, there are no symmetries;
  \item or \emph{unlabeled}: in which case, an atom is indistinguishable from
    the next, and so certain symmetries must be taken into account (so
    that two objects which are not decomposed in the same way but have the
    same ultimate shape are not counted twice).
  \end{itemize}
  \noindent It is often the case that enumerations for labeled classes are
  easier to obtain than for unlabeled ones. Our grammars allow to derive
  generating functions, enumerations, and asymptotics for both.
\end{remark}


\section{3-Leaf Power Graphs\label{sec:3lp}}


The first class that we discuss is that of $3$-leaf power graphs: a
chordal subset of distance-hereditary graphs\footnote{Not a maximal such
  subset, as it is known that \emph{ptolemaic graphs} are the intersection
  of chordal graphs and distance-hereditary graphs.}.

\begin{definition}
  A graph $G=(V,E)$ is a $k$-leaf power graph\footnote{This is a
    specialization, introduced by Nishimura~\etal~\cite[\S 1]{NiRaTh02b},
    of the concept of \emph{graph powers}, in which the root is a
    tree---but the definition can be extended to the case where $T$ is not
    a tree, but is a graph $H$ (in which case, we consider the distance
    between any two vertices in graph $H$, not two leaves of a tree).}, if
  there is a tree $T$ (called a $k$-leaf root of graph $G$) such that:
  \begin{enumerate}[label=(\alph*), noitemsep, nosep]
  \item the leaves of $T$ are the vertices $V$;
  \item there is an edge $xy\in E$ if and only if the distance in $T$
    between leaves $x$ and $y$ is at most $k$,
    $\cramped{d_T}(x,y)\leqslant k$.
  \end{enumerate}
\end{definition}

\noindent These families of graphs are relevant to
phylogenetics~\cite{NiRaTh02b}: from the the pairwise genetic distance
between a collection of species (which is a graph), it is desirable to
establish a tree which highlights the most likely ancestry (or more
broadly, the evolutionary relationships) relations between species.

We begin with the enumeration of 3-leaf power graphs, the smaller
combinatorial class, because the application of the dissymmetry theorem
(used to obtain an enumeration of the unrooted class given the grammar for
some rooted version of the class) in Subsection~\ref{subsec:unroot-3leaf}
is less involved for 3-leaf power graphs than it is for
distance-hereditary graphs.

\subsection{Grammar\protect\footnote{All grammars that we produce in this
    article yield an incorrect enumeration for the first two terms (graphs
    of size 1 and 2), because Cunningham's Theorem, presented in
    Subsection~\ref{subsec:split} requires non-leaf nodes to have degree
    at least three: thus the special cases of graphs involving only 1 or 2
    nodes must be treated non-recursively. While we could amend the
    grammars accordingly, we think it would be less elegant---especially
    since there is generally little confusion regarding those first few
    terms.\label{footnote:no-singletons}} from the split-decomposition.}

The starting point is the characterization of the split-decomposition tree
of 3-leaf power graphs, as introduced by Gioan and Paul~\cite{GiPa12}.

\begin{theorem*}[Characterization of $3$-leaf power split-decomposition
  tree~{\cite[\S~3.3]{GiPa12}}]
  A connected graph $G=(V,E)$ is a $3$-leaf power graph if and only if:
  \begin{enumerate}[label=(\alph*), noitemsep, nosep]
  \item its split-decomposition tree $ST(G)$ is a clique-star tree
    (implies that $G$ is distance-hereditary);
  \item the set of star-nodes forms a connected subtree of $T$;
  \item the center of a star-node is incident either to a leaf or a clique node.
  \end{enumerate}
\end{theorem*}

\noindent This is unsurprising given that an alternate (perhaps more
pertinent) characterization is that a 3-leaf power graph can be obtained
from a tree by replacing every vertex by a clique of arbitrary size.

\begin{theorem}%
  \label{thm:3lp-rooted-grammar}%
  The class ${\clsTLPrl}$ of $3$-leaf power graphs rooted at a
  vertex\footnote{Or, equivalently, rooted at a leaf of its
    split-decomposition tree.} is specified by
  \begin{align}
    \clsTLPrl       &=\cls[\bullet]{L}\times\left(\cls[C]{S}+\cls[X]{S}\right)+\cls[\mLeaf]{K}\\
    \cls[C]{S}      &= \Set[\geqslant 2]{\cls{L}+\cls[X]{S}}\\
    \cls[X]{S}      &= {\cls{L}}\times\Set[\geqslant 1]{\cls{L} + \cls[X]{S}}\label{eq:3lp-root-sx}\\
    \cls{L}         &= \clsAtom + \Set[\geqslant 2]{\clsAtom}\\
    \cls[\mLeaf]{L} &= \clsAtom_{\mLeaf} + \clsAtom_{\mLeaf}\times\Set[\geqslant 1]{\clsAtom}\\
    \cls[\mLeaf]{K} &= \clsAtom_{\mLeaf}\times\Set[\geqslant 2]{\clsAtom}\text{.}
  \end{align}
\end{theorem}

\noindent In this combinatorial specification, we define several classes
of subtrees: we denote by $\cramped{\cls[X]{S}}$ (resp.
$\cramped{\cls[C]{S}}$) the class of split-decomposition trees rooted at a
star-node which are \emph{linked to their parent} by an extremity of this
star-node (resp. the center of this star-node).

Finally, because the structure of the split-decomposition tree of a 3-leaf
power graph only allows for cliques that are incident to at most one
star-node (and the rest of the edges must lead to leaves), we have three
classes $\cls{L}$, $\cramped{\cls[\mLeaf]{L}}$ and
$\cramped{\cls[\mLeaf]{K}}$ which express leaves and cliques\footnote{The
  class $\cls{L}$ is a class containing either a leaf; or a clique-node
  connected to all but one of its extremities to leaves. The class
  $\cramped{\cls[\mLeaf]{L}}$ is that same class, in which one of the
  leaves has been distinguished (as the root of the tree).}.

\begin{proof}
  In addition to the constraints specific to 3-leaf power
  split-decomposition trees given in the characterization above, because
  the split-decomposition trees we are enumerating are \emph{reduced} (see
  Cunningham's Theorem in Section~\ref{subsec:split}), there are two
  additional implicit constraints on their internal nodes:
  \begin{itemize}
  \item the center of a star cannot be incident to the extremity of
    another star (because then they would be merged with a
    \emph{star-join} operation, as in Figure~\ref{fig:star-join}, yielding
    a more concise split-decomposition tree);
  \item and two cliques may not be incident (or they would be merged with
    a \emph{clique-join} operation, as in Figure~\ref{fig:clique-join}).
  \end{itemize}
  
  \noindent The star-nodes form a connected subtree, with each star-node
  connected to others through their extremities; the centers are
  necessarily connected to ``leaves'', and the extremities may be
  connected to ``leaves''; ``leaves'' are either single nodes (actually
  leaves) or cliques (which are a set of more than two elements, because
  cliques have minimum size of 3 overall, including the parent
  node).\medskip

  \noindent First, the following equation
  \begin{align*}
    \cls[C]{S}      &= \Set[\geqslant 2]{\cls{L}+\cls[X]{S}}
  \end{align*}
  indicates that a subtree rooted at a star-node, linked to its parent
  (presumably a leaf) by its center, is a set of size at least 2 children,
  which are the extremities of the star-node: each extremity can either
  lead to a ``leaf'' or to another star-node entered through an extremity.

  Next, the equation
  \begin{align*}
    \cls[X]{S}      &= {\cls{L}}\times\Set[\geqslant 1]{\cls{L} + \cls[X]{S}}
  \end{align*}
  indicates that a subtree rooted at a star-node, linked to its parent by
  an extremity, is the Cartesian product of a ``leaf'' (connected through
  the center of the star-node) and a set of 1 or more children which are
  the extremities of the star-node: each leads either to a``leaf'' or to
  another star-node entered through an extremity.

  The ``leaves'' are then either an actual leaf of unit size, or a clique;
  the clique has to be of size at least 3 (including the incoming link)
  and the children can only be actual leaves. We are thus left with
  \begin{align*}
    \cls{L}         &= \clsAtom + \Set[\geqslant 2]{\clsAtom}\text{.}
  \end{align*}
  Finally, the rest of the grammar deals with the special cases that arise
  from when the split-decomposition tree does not contain any star-node at
  all.
\end{proof}

\noindent With the grammar for ${\clsTLPrl}$, we are able to produce the
exact enumeration for labeled rooted 3-leaf power graphs, and by a simple
algebraic trick, for unlabeled rooted 3-leaf power graphs.

\begin{corollary}[Enumeration of labeled $3$-leaf power graphs]
  Let $T(z)$ be the exponential generating function associated with the
  class ${\clsTLPrl}$. Then, the enumeration of labeled, \emph{unrooted}
  $3$-leaf power graphs, for $n\geqslant 3$, is given by
  \begin{align}
    t_n = (n-1)! [\cramped{z^n}] T(z)\text{,}
  \end{align}
  to the effect that the first few terms of the enumeration, \EIS{A00000}, are
  \begin{align*}
    &1, 1, 4, 35, 361, 4482, 68027, 1238841, 26416474,\\
    &\qquad 646139853, 17837851021, 548713086352, \ldots
  \end{align*}
\end{corollary}

\subsection{Unrooting unlabeled objects\label{subsec:unrooting}.}

The trees described by the specification of ${\clsTLPrl}$ have leaves
which are labeled, one of which is the root. Thus because each label has
equal opportunity of being the root, it is simple to obtain an enumeration
of the labeled \emph{unrooted} class by dividing by $n$.

When now considering \emph{unlabeled} trees, however, proceeding in this
way leads to an overcount of certain trees, because of new symmetries
introduced by the disappearance of labels. Fortunately, we can use the
\emph{dissymmetry theorem for trees}, which expresses the enumeration of
an unrooted class of trees in terms of the enumeration of the equivalent
rooted class of trees.

This theorem was introduced by Bergeron~\etal~\cite{BeLaLe98} in terms of
ordered and unordered pairs of trees, and was eventually reformulated in a
more elegant manner, such as in Flajolet and Sedgewick~\cite[VII.26
p.~481]{FlSe09} or Chapuy~\etal~\cite[\S 3]{ChFuKaSh08}. It states
\begin{align}\label{eq:dissymmetry}
  \cls{A} + {\cls{A}}_{\mDEdge} \simeq
  {\cls{A}}_{\mNode} + {\cls{A}}_{\mEdge}
\end{align}
where $\cls{A}$ is the unrooted class of trees, and
$\cramped{{\cls{A}}_{\mNode}}$, $\cramped{{\cls{A}}_{\mEdge}}$,
$\cramped{{\cls{A}}_{\mDEdge}}$ are the rooted class of trees respectively
where only the root is distinguished, an edge from the root is
distinguished, and a directed, outgoing edge from the root is
distinguished\footnote{Drmota~\cite[\S 4.3.3, p.~293]{HoECDrmota15}
  presents an elegant proof of this result by appealing to the notion of
  \emph{center} of the tree---which may be a single vertex or an edge;
  indeed, Drmota builds a bijection between the trees
  $\cramped{\cls[\mathrm{nc}]{A}}$ rooted at a non-central vertex/edge and
  trees rooted at a directed edge, by orienting the root of the first
  class in the direction of the center.}.

The application of this theorem may initially be perplexing, and so we
begin by making a couple of remarks.

\begin{lemma}\label{lem:no-leaves}
  In the dissymmetry theorem for trees, when rerooting at the nodes (or
  atoms) of a combinatorial tree-like class $\cls{A}$, leaves can be
  ignored.
\end{lemma}

\begin{proof}
  When we point a node of the class $\cls{A}$, we may distinguish whether
  it is an internal node or a leaf, which we respectively denote $\mINode$
  and $\mLeaf$ in the \emph{right hand side} of the following equations.
  Accordingly,
  \begin{align*}
    {\cls{A}}_{\mDEdge} &= {\cls{A}}_{\mLDEdge} + {\cls{A}}_{\mDIEdge} + {\cls{A}}_{\mDLEdge}\\
    {\cls{A}}_{\mEdge}  &= {\cls{A}}_{\mLEdge} + {\cls{A}}_{\mIEdge}\\
    {\cls{A}}_{\mNode}  &= {\cls{A}}_{\mLeaf} + {\cls{A}}_{\mNode}
  \end{align*}
  where the first equation should be understood as: if we mark a directed
  edge of the class $\cls{A}$, it can either go from an internal node to a
  leaf, from a leaf to an internal node, or from an internal node to
  another internal node\footnote{We are ignoring the very special case of
    a tree reduced to an edge, in which we may have an edge between two
    leaves; this explains why our unrooted grammars may, if uncorrected,
    be wrong for the first two terms. This is analogous to the initial
    term errors of our rooted grammars, as expressed in
    Footnote~\ref{footnote:no-singletons}.}.

  These equations may be further simplified upon observing that any edge
  of which one of the endpoints is a leaf, is entirely determined by that
  leaf, to the effect that
  \begin{align*}
    {\cls{A}}_{\mLeaf} = {\cls{A}}_{\mLEdge} =
    {\cls{A}}_{\mLDEdge}\text{.}
  \end{align*}
  Thus proving that one may disregard leaves entirely when applying the
  dissymmetry theorem for trees.
\end{proof}

\begin{remark}
  While the dissymmetry theorem considers pointed internal nodes, our
  grammars ${\clsTLPrl}$ and ${\clsDHrl}$ (respectively derived from the
  split-decomposition of 3-leaf power graphs and distance-hereditary
  graphs) are pointed at the \emph{leaves} of the split-decomposition tree
  (which correspond to the vertices of the original graph).
  
  This is not, in fact, a discrepancy. When we apply the dissymmetry
  theorem, we implicitly \emph{reroot} the trees from our grammars at
  internal nodes, which we express as subclasses $\cramped{\cls[x]{T}}$ of
  trees rooted in some specific type of internal node $x$. Rerooting an
  already rooted tree is relatively easy (while unrooting a rooted tree is
  not!).
\end{remark}

\begin{remark}
  The dissymmetry theorem establishes a bijection between two disjoint
  unions; this allows us to recover an equation on the coefficients,
  \begin{align}
    \begin{split}
      [\cramped{z^n}] A(z) &= [\cramped{z^n}] A_{\mNode}(z)\\
                           &+ [\cramped{z^n}] A_{\mEdge}(z)\\
                           &- [\cramped{z^n}] A_{\mDEdge}(z).
    \end{split}
  \end{align}
  However the subtraction has no combinatorial meaning, which means that
  once the dissymetry theorem has been applied, we lose the symbolic
  meaning of the equation.

  While this is enough to compute exact enumerations (by extracting the
  enumeration of each generating function and algebraically computing the
  equation), and is sufficient to deduce some asymptotics, there is not
  enough information, for instance, to yield a recursive
  sampler~\cite{FlZiVa94} or a Boltzmann
  sampler~\cite{DuFlLoSc04,FlFuPi07}---and we are instead left with ad-hoc
  methods to generate unrooted objects, see Iriza~\cite[\S~3.2]{Iriza15}.

  Unrooting the initial grammar while preserving the symbolic nature of
  the specification requires using a more complex combinatorial tool
  called \emph{cycle-pointing}\footnote{This operation, given a structure
    of size $n$, finds $n$ ways to group its atoms/vertices in cycles
    which mirror the symmetries of the structure. This is analogous to
    atom/vertex-pointing in labeled objects, where each structure of size
    $n$ can be pointed $n$ different ways (each atom/vertex can be pointed
    because they are each distinguishable and there are no symmetries that
    would make two different pointings equivalent).} introduced by
  Bodirsky~\etal~\cite{BoFuKaVi11}, and applied to these grammars by
  Iriza~\cite[\S 5.5]{Iriza15}, it has allowed us to generate the random
  graphs provided in figures to this article.
\end{remark}

\subsection{Applying the dissymmetry theorem.\label{subsec:unroot-3leaf}}

\begin{theorem}%
  \label{thm:3lp-unrooted-grammar}%
  The class ${\clsTLP}$ of unrooted $3$-leaf power graphs is specified by
\begin{align}
  \clsTLP                &= \cls{K}+\cls[S]{T} + \cls[S-S]{T} - \cls[S\rightarrow S]{T}\\[0.4em]
  \cls[S]{T}             &= \cls{L}\times \cls[C]{S}\\
  \cls[S-S]{T}           &= \Set[2]{\cls[X]{S}}\\
  \cls[S\rightarrow S]{T}&= \cls[X]{S}\times\cls[X]{S}\\[0.4em]
  \cls[C]{S}             &= \Set[\geq 2]{\cls{L}+\cls[X]{S}}\\
  \cls[X]{S}             &= {\cls{L}}\times\Set[\geqslant 1]{\cls{L} + \cls[X]{S}}\\
  \cls{L}                &= \clsAtom + \Set[\geqslant 2]{\clsAtom}\\
  \cls{K}                &= \Set[\geq 3]{\clsAtom}\text{.}
\end{align}

\end{theorem}

\begin{proof}
  From the dissymmetry theorem, we have the symbolic equation linking the
  rooted and unrooted decomposition tree of 3-leaf power graphs,
  \begin{align*}
    \clsTLP &= {{\clsTLP}_{\mNode}} + 
               {{\clsTLP}_{\mEdge}} -
               {{\clsTLP}_{\mDEdge}}\text{.}
  \end{align*}
  As per Lemma~\ref{lem:no-leaves}, it suffices to consider only internal
  nodes, and the only type of internal node found in these
  split-decomposition trees is the star-node\footnote{In the
    split-decomposition of a 3-leaf power graph, clique-nodes cannot have
    any children other than leaves; as a result, they may be considered as
    leaves for the purpose of the dissymmetry theorem.}.
  
  So we must reroot the grammar ${\clsTLPrl}$, which is rooted at a leaf
  of the split-decomposition tree, to each of: a star-node, an undirected
  edge connecting two star-nodes, and a directed edge connecting two
  star-nodes.
  
  Rerooting at a star-node, we must consider all the outgoing edges of the
  star. The center will lead either to a leaf, or to a clique---this is
  the rule $\cls{L}$; what remains are then the extremities, which can be
  expressed by the term $\cramped{\cls[C]{S}}$. Since the center is
  distinguished, this is combined as a Cartesian product, hence
  \begin{align*}
    \cls[S]{T}             &= \cls{L}\times \cls[C]{S}\text{.}
  \end{align*}
  Next, we reroot at an edge. Since these split-decomposition trees are
  reduced, two star-nodes can only be adjacent at their respective
  centers, or at two extremities; but because of the additional constraint
  for 3-leaf power graphs, two star-nodes can only be adjacent at their
  extremities.
  
  Rerooting at an undirected edge will yield a set containing two
  elements; rerooting at a directed edge will yield a Cartesian product.
  Thus, we have
  \begin{align*}
    \cls[S-S]{T}           &= \Set[2]{\cls[X]{S}}\\
    \cls[S\rightarrow S]{T}&= \cls[X]{S}\times\cls[X]{S}\text{.}
  \end{align*}
  Finally, as with the original vertex-rooted grammar ${\clsTLPrl}$, we
  must deal with the special case of a graph reduced to a clique, as it
  does not involve any star-node.
\end{proof}

\tmpsetcounter{corollary}{2}
\begin{corollary}[Enumeration of unlabeled, unrooted $3$-leaf power
  graphs]
  The first few terms of the enumeration, \EIS{A00000}, are
  \begin{align*}
      &1, 1, 2, 5, 12, 32, 82, 227, 629, 1840, 5456, 16701,\\
      &\qquad 51939, 164688, 529070, 1722271, 5664786,\\
      &\qquad 18813360, 62996841, 212533216\ldots
  \end{align*}
\end{corollary}
\tmprestorecounter{corollary}

\section{Distance-Hereditary Graphs\label{sec:dh}}


A graph is \emph{totally decomposable} by the split-decomposition if every
induced subgraph with at least 4 vertices contains a split. And it is
well-known~\cite{HaMa90} that the class of totally decomposable graphs is
exactly distance-hereditary graphs.

Deriving the rooted grammar provided in
Theorem~\ref{thm:dh-rooted-grammar} is easier than for 3-leaf power
graphs, because there are few constraints on the split-decomposition tree
of distance-hereditary graphs; as a result, applying the dissymmetry
theorem will be a bit more involved because there are two types of
internal nodes at which to reroot the tree.

\begin{theorem}%
  \label{thm:dh-rooted-grammar}%
  The class $\clsDHrl$ of distance-hereditary graphs rooted at a vertex is specified by
  \begin{align}
    \clsDHrl    &= \clsAtom_{\mLeaf}\times\left(\cls{K} + \cls[C]{S} + \cls[X]{S}\right)\\
    \cls{K}     &= \Set[\geqslant 2]{\clsAtom + \cls[C]{S} + \cls[X]{S}}\label{eq:dhrl-c}\\
    \cls[C]{S}  &= \Set[\geqslant 2]{\clsAtom + \cls{K} + \cls[X]{S}}\label{eq:dhrl-sc}\\
    \cls[X]{S}  &= \Seq[\geqslant 2]{\clsAtom + \cls{K} + \cls[C]{S}}\label{eq:dhrl-sx}\text{.}
  \end{align}
\end{theorem}

\begin{proof}
  We describe a grammar for clique-star trees subject only to the
  irreducibility constraint: a star's center cannot be connected to the
  extremity of another star (see Figure~\ref{fig:star-join}), and two
  cliques cannot be connected (see Figure~\ref{fig:clique-join}).

  We start with the following rule
  \begin{align*}
    \clsDHrl    &= \clsAtom_{\mLeaf}\times\left(\cls{K} + \cls[C]{S} + \cls[X]{S}\right)
  \end{align*}
  in which $\cramped{\clsAtom_{\mLeaf}}$, the vertex at which the
  split-decomposition tree is rooted, can be connected either to a clique
  $\cls{K}$, or to a star's extremity $\cramped{\cls[X]{S}}$, or to a
  star's center $\cramped{\cls[C]{S}}$.

  Next, we describe subtrees rooted at a clique
  \begin{align*}
    \cls{K}     &= \Set[\geqslant 2]{\clsAtom + \cls[C]{S} + \cls[X]{S}}\text{,}
  \end{align*}
  we are connected to our parent by one of the outgoing edges of the
  clique, and because clique-nodes have size at least 3 (see Cunningham's
  Theorem in Subsection~\ref{subsec:split} which requires non-leaf nodes
  to have degree at least 3), we are left with at least two subtrees to
  describe:
  \begin{itemize}
  \item these subtrees can either be a leaf $\clsAtom$, or a star entered
    either by its center $\cramped{\cls[C]{S}}$ or its extremity
    $\cramped{\cls[X]{S}}$---they cannot be another clique because our
    tree could then be reduced with a clique-join operation;
  \item because of the symmetries within a clique (in particular there is
    no ordering of the vertices), the order of the subtrees does not
    matter, and so these are described by a $\Set$ operation.
  \end{itemize}
  
  \noindent By similar arguments, we describe subtrees rooted at a star
  which is connected to its parent by its center,
  \begin{align*}
    \cls[C]{S}  &= \Set[\geqslant 2]{\clsAtom + \cls{K} + \cls[X]{S}}\text{.}
  \end{align*}
  Because the star's center is connected to its parent, we only need
  express what the extremities are connected to; each of these can be
  connected to a leaf, a clique, or another star by one of that star's
  extremity (to avoid a star-join). Again, as the extremities are
  indistinguishable from each other---the star is not planar---we describe
  the subtrees by a $\Set$ operation.
  
  We are left with the subtrees rooted at star which is connected to its
  parent by an extremity; these may be described by
  \begin{align*}
    \cls[X]{S}  &= \left(\clsAtom +\cls{K}+\cls[C]{S}\right)\times
                      \Set[\geqslant 1]{\clsAtom + \cls{K} + \cls[X]{S}}\text{.}
  \end{align*}
  Indeed, the first term of the Cartesian product is the subtree to which
  the center is connected (either a leaf, a clique, or another star at its
  center); the $\Set$ expresses the remaining extremities---of which there
  is at least one. This equation can be simplified to obtain the one in
  the Theorem---but this simplification is proven in
  Appendix~\ref{app:dh-simplification}.
\end{proof}

\begin{remark}
  We notice the same symbolic rules for the clique-node $\cls{K}$ and the
  star-node $\cramped{\cls[C]{S}}$ entered through the center,
  respectively in Equations~\eqref{eq:dhrl-c} and~\eqref{eq:dhrl-sc}. This
  suggest these nodes play a symmetrical role in the overall grammar, and
  that their associated generating function (and enumeration) are
  identical.

  It would be mathematically correct to merge both rules, \textit{e.g.}
  \begin{align}
    \clsDHrl &= \clsAtom_{\mLeaf}\times\left(\cls{K} +
                {\cls{K}} + \cls{S}\right)\\
    \cls{K}  &= \Set[\geqslant 2]{\clsAtom + {\cls{K}} + \cls{S}}\label{eq:dhrs-k}\\
    \cls{S}  &= \Seq[\geqslant 2]{\clsAtom + \cls{K} + {\cls{K}}}\label{eq:dhrs-s}\text{.}
  \end{align}

  \noindent This may be convenient (and lead to additional
  simplifications) for some usages, such as the application of asymptotic
  theorems like those introduced by Drmota~\cite{Drmota97} (see
  Section~\ref{sec:asympt}) which requires classes be expressed as a
  single functional equation.

  However the combinatorial meaning of the symbols is lost: in the above
  system, it can no longer be said that $\cls{K}$ represents a
  clique-node. This is problematic for parameter analysis (\textit{e.g.},
  if trying to extract the average number of clique-nodes in the
  split-tree of a uniformly drawn distance-hereditary graph).
\end{remark}

\begin{theorem}%
  \label{thm:dh-unrooted-grammar}%
  The class $\clsDH$ of unrooted distance-hereditary graphs is specified by
  \begin{align}
    \clsDH                 &= \cls[K]{T} + \cls[S]{T} + \cls[S-S]{T}
                             - \cls[K - S]{T} - \cls[S\rightarrow S]{T}
                              \label{eq:dhu-main}\\[0.4em]
    \cls[K]{T}             &=\Set[\geqslant 3]{\clsAtom + \cls[C]{S} + \cls[X]{S}}\\
    \cls[S]{T}             &=\left(\clsAtom+\cls{K}+\cls[C]{S}\right)\times \cls[C]{S}\\
    \cls[K-S]{T}           &=\cls{K}\times\left(\cls[C]{S} + \cls[X]{S}\right)\\
    \cls[S-S]{T}           &=\Set[2]{\cls[C]{S}} + \Set[2]{\cls[X]{S}}\\
    \cls[S\rightarrow S]{T}&=\cls[C]{S}\times\cls[C]{S}+ \cls[X]{S}\times\cls[X]{S}\\[0.4em]
    \cls{K}     &= \Set[\geqslant 2]{\clsAtom + \cls[C]{S} + \cls[X]{S}}\\
    \cls[C]{S}  &= \Set[\geqslant 2]{\clsAtom + \cls{K} + \cls[X]{S}}\\
    \cls[X]{S}  &= \Seq[\geqslant 2]{\clsAtom + \cls{K} + \cls[C]{S}}\text{.}
  \end{align}
\end{theorem}

\begin{proof}
  This is again an application of the dissymmetry theorem for trees, and
  as before, we may ignore the leaves, and mark only the internal nodes,
  \begin{align*}
    \clsDH &= {{\clsDH}_{\mNode}} + 
              {{\clsDH}_{\mEdge}} -
              {{\clsDH}_{\mDEdge}}\text{.}
  \end{align*}
  Unlike for 3-leaf power graphs in Subsection~\ref{subsec:unroot-3leaf},
  the tree decomposition of distance-hereditary graphs clearly involves
  two types of internal nodes: cliques and stars. If we express all the
  rerooted trees we will have to express, we get the expression:
  \begin{align}
    \begin{split}
      \clsDH &= \cls[K]{T} + \cls[S]{T}\\
             &+ \cls[S-S]{T} + \cls[S-K]{T}\\
             &- \cls[S\rightarrow S]{T} - \cls[S\rightarrow K]{T}
              - \cls[K\rightarrow S]{T}\text{.}
    \end{split}
  \end{align}
  \noindent Note that we do not have a tree rerooted at an edge involving
  two cliques, because as mentioned previously, the split-decomposition
  tree would not be reduced, since the two cliques could be merged with a
  clique-join.

  A first simplification can be made, because a directed edge linking two
  internal nodes of different type is equivalent to a non-directed edge,
  because the nature of the two internal nodes already distinguishes them,
  thus in particular here
  \begin{align*}
    \cls[K\rightarrow S]{T} \simeq \cls[K-S]{T}\text{.}
  \end{align*}
  In doing so, several terms cancel out, which leads us to:
  \begin{align*}
      \clsDH &= \cls[K]{T} + \cls[S]{T} + \cls[S-S]{T}
                - \cls[K-S]{T} - \cls[S\rightarrow S]{T}\text{.}
  \end{align*}
  We then only have to express the rerooted classes:
  \begin{enumerate}[leftmargin=1.8cm, labelsep=0.5cm]
  \item[{$\cls[K]{T}$}] For a clique-node, we must account for at least
    three outgoing edges, which can be connected to anything besides
    another a clique-node.

  \item[{$\cls[S]{T}$}] For a star-node, we reuse the same trick as
    previously: we express what the center can be connected to (either a
    leaf, a clique-node or the center of another star-node), and then we
    use $\cramped{\cls[C]{S}}$ to express the remaining extremities, as
    explained in the unrooted grammar for the 3-leaf power graphs.\medskip

  \item[{$\cls[K-S]{T}$}] The undirected edge already accounts for a
    connection between a clique-node and a star-node, so we must describe
    the remaining outgoing edges of these two combined nodes: for the
    clique, this can be expressed by reusing the subtree $\cls{K}$ (which
    is exactly a tree rooted a clique which is missing one subtree---the
    one connected to the star-node); for the star, if it is connected to
    the clique-node by its extremity, we can use $\cramped{\cls[X]{S}}$,
    otherwise $\cramped{\cls[C]{S}}$.

  \item[{$\cls[S-S]{T}$}] Two star-nodes can only be connected at two of
    their extremities, or their respective centers\footnote{For the 3-leaf
      power graphs, we only considered two star-nodes connected at two of
      their extremities, because part of the characterization of 3-leaf
      power graphs is that the center of stars are oriented away from
      other stars.}; because the edge is undirected, we use a $\Set$
    operation.\medskip

  \item[{$\cls[S\rightarrow S]{T}$}] Same as above, except the edge being
    now directed, we use a Cartesian product to express that there is a
    source star-node and a destination star-node.
  \end{enumerate}
\end{proof}

\begin{figure*}
  \centering
  \includegraphics[scale=0.6]{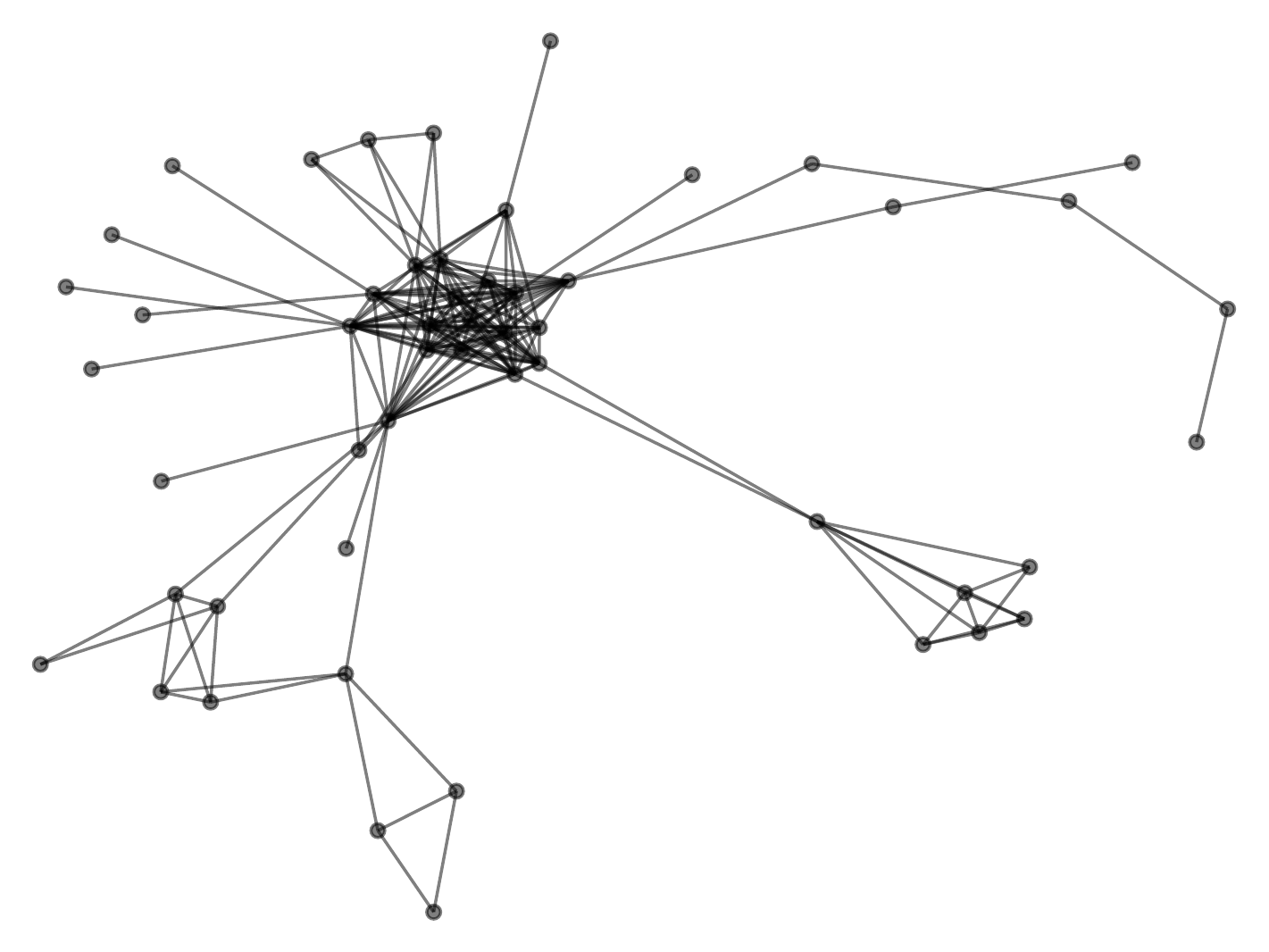}
  \caption{\label{fig:random-dh-52}A randomly generated
    distance-hereditary graph with 52 vertices, produced using the
    Boltzmann samplers developed by Iriza~\cite{Iriza15}.}
\end{figure*}

\section{Asymptotics\label{sec:asympt}}

\def\cS{\mathcal{S}}
\def\cK{\mathcal{K}}
\def\cL{\mathcal{L}}
\def\cT{\mathcal{T}}
\def\cZ{\mathcal{Z}}
\def\ni{\noindent}

In this section, using singularity analysis of generating functions, we
derive asymptotic estimates for the number of \emph{unlabeled} (rooted and
unrooted) 3-leaf power graphs and distance-hereditary graphs (with respect
to the number $n$ of vertices).
 
The strategy in both cases is very similar. The first step is to deal with
the rooted case, where the decomposition grammar
(Theorem~\ref{thm:3lp-rooted-grammar} for 3-leaf power graphs,
Theorem~\ref{thm:dh-rooted-grammar} for distance-hereditary graphs)
translates into an equation system for the corresponding generating
functions. The system is in fact sufficiently simple that, using suitable
manipulations, it can be reduced to a single-line equation of the form
$y=F(y,z)$, where $y$ is one of the rooted generating functions and all
the other ones have a simple expression in terms of $y$. The
Drmota-Lalley-Wood theorem then ensures that the rooted generating
functions classically have a square-root singularity, yielding asymptotic
estimates of the form $c\ \!\cramped{\rho^{-n}}\cramped{n^{-3/2}}$ (see
also Drmota~\cite{Drmota97}).

The next step is to study the generating function $U(z)$ for unrooted 3LP
(resp. DH) graphs. From the dissymmetry theorem
(Theorem~\ref{thm:3lp-unrooted-grammar} for 3LP graphs,
Theorem~\ref{thm:dh-unrooted-grammar} for DH graphs) we obtain an
expression of $U(z)$ in terms of $y$ and $z$, from which we can obtain a
singular expansion of $U(z)$. As expected, the subtractions involved in
the expression of $U(z)$ yield a cancellation of the square-root terms, so
that the leading singular terms are at the next order, yielding asymptotic
estimates of the form $d\ \!\rho^{-n}n^{-5/2}$ (which are usual for
unrooted ``tree-like'' structures).

\begin{remark}
  A similar approach has been previously applied to another tree
  decomposition of graphs (decomposition into 2-connected blocks and a
  tree to describe the adjacencies between blocks). Based on this
  decomposition, the generating functions for several families of graphs
  have been obtained (cacti graphs, outerplanar graphs~\cite{BoFuKaVi07b},
  series-parallel graphs~\cite{DrFuKaKrRu11}), along with asymptotics of
  the form $d\ \!\cramped{\rho^{-n}}\cramped{n^{-5/2}}$.
\end{remark}

\subsection{The case of 3-leaf power graphs.\label{sec:3LP_asymp}}

We start with the \emph{rooted} case. Let $L(z)$ and $\cramped{S_X}(z)$ be
the generating functions of $\cls{L}$ and $\c[X]{S}$; note that after
simplification (and only in the unlabeled case) $L(z)=z/(1-z)$. Then
Eq.~\eqref{eq:3lp-root-sx} in Theorem~\ref{thm:3lp-rooted-grammar} yields
\begin{align*}
S_X(z)=L(z)\cdot\Big(\exp\big(\sum_{i\geq 1}\frac1{i}(L(z^i)+S_X(z^i))\big)-1\Big).
\end{align*} 
Hence $y=S_X(z)$ satisfies the functional equation
\begin{align*}
y=\frac{z}{1-z}\cdot\big(\exp(y+B(z))-1\big),
\end{align*}
where
\begin{align*}
  B(z):=\sum_{i\geq 1}\frac{1}{i}\frac{z^i}{1-z^i}+\sum_{i\geq 2}\frac{1}{i}S_X(z^i)\text{.}
\end{align*}
This is a functional equation of the form $y=F(z,y)$, with $F(z,y)$ a
bivariate formal power series with nonnegative coefficients and that has
nonlinear dependence on $y$. Let $\rho$ be the radius of convergence of
$S_X(z)$. The fact that $F(z,y)$ is superlinear in $y$ ensures that
$S_X(z)$ converges to a finite positive value (denoted by $\tau$) when $z$
tends to $\rho$ from below. Furthermore it is easy to check
combinatorially that $[z^n]S_X(z)$ is of exponential growth
(\textit{i.e.}, that there exists $\alpha>1$ such that
$[z^n]S_X(z)\geq \alpha^n$ for $n$ large enough), thus $\rho<1$. This
easily implies that $B(z)$ is analytic at $\rho$, hence $F(z,y)$ is
analytic at $(\rho,\tau)$.

The conditions of the Drmota-Lalley-Woods theorem~\cite[Thm~VII.6
p.~489]{FlSe09} hold: the singularity of $S_X(z)$ has to be due to a
branch-point, \textit{i.e.}, we have $F_y(z,y)=1$ at $(z,y)=(\rho,\tau)$,
and $S_X(z)$ has a singular expansion around $\rho$ of the form
$$
S_X(z)=\tau-c\cdot Z+O(Z^2),\ \ \mathrm{where}\ Z=\sqrt{1-z/\rho},
$$ 
where $c=(2\rho\ \!F_z(\rho,\tau)/F_{yy}(\rho,\tau))^{1/2}$. 
Moreover $S_X(z)$ has an analytic continuation to a $\Delta$-domain of the form $\{|z|\leq \rho+\epsilon\}\cap\{z-\rho\notin \mathbb{R}_+\}$. Hence we can apply classical transfer theorems
to obtain $[z^n]S_X(z)\sim\frac{c}{\sqrt{2\pi}}\rho^{-n}n^{-3/2}$. 

\vspace{.4cm}

\ni In order to evalutate $\rho$, we have to solve the system
$\{y=F(z,y),1=F_y(z,y)\}$. There is a difficulty due to the fact that
$F(z,y)$ involves quantities $S(z^i)$ for $i\geq 2$. Following Flajolet
and Sedgewick~\cite[\S VII.5]{FlSe09}, we can however accurately
approximate these quantities by $S^{[m]}(z^i)$, where $S^{[m]}(z)$ is the
polynomial of degree $m$ coinciding with the Taylor expansion of $S(z)$ to
order $m$. Denoting by $F^{[m]}(z,y)$ the corresponding (now explicit)
approximation of $F(z,y)$, we can solve for the system
$\{y=F^{[m]}(z,y),1=F^{[m]}_y(z,y)\}$; and the obtained solution
$(\rho^{[m]},\tau^{[m]})$ is found to converge exponentially fast when $m$
increases; we find $\rho\approx 3.848442876\ldots$.

\vspace{.4cm}


We now move to the asymptotic enumeration of unrooted 3LP graphs. For that
purpose we will express the generating function $U(z)$ of unrooted 3LP
graphs in terms of $S_X(z)$ and will show that the leading singular term
is of order $Z^3$ due to a cancellation of the coefficients for terms of
order $Z$. We have to deal here with singular expansions up to order
$Z^3$, and a first important point is that (as an application of the
Drmota-Lalley-Woods theorem) $S_X(z)$ admits such an expansion, of the
form
$$
S_X(z)=\tau-c\cdot Z+d\cdot Z^2+e\cdot Z^3+O(Z^3).
$$
Let $U(z)$ be the generating function of unrooted 3LP graphs according
to the number of vertices. It follows from the grammar of Theorem 2 that 
$$
U(z)=K(z)+T_S(z)+T_{S-S}(z)-T_{S\to S}(z),
$$
We have $K(z)=z^3/(1-z)$. In order to express $T_S(z)$ in terms of
$S_X(z)$, we note that
\begin{align*}
  \cls[X]{S} &=\cls{L}\times\Set[\geq 1](\cls{L}+\cls[X]{S})\\
             &=\cls{L}\times\Set[\geq 2](\cls{L}+\cls[X]{S})
              +\cls{L}\times(\cls{L}+\cls[X]{S})\\
             &=\cls[S]{T}+\cls{L}\times(\cls{L}+\cls[X]{S})\text{,}
\end{align*}
hence
\begin{align*}
  T_S(z)&=S_X(z)-L(z)\cdot(L(z)+S_X(z))\\
        &=S_X(z)-\frac{z}{1-z}\cdot\Big(\frac{z}{1-z}+S_X(z)\Big)\text{.}
\end{align*}
Finally we have
\begin{align*}
  T_{S-S}(z)=\frac{S_X(z^2)+S_X(z)}{2}%
  \quad\text{and}\quad
  T_{S\to S}(z)=S_X(z)^2\text{,}
\end{align*}
so that we obtain
\begin{align*}
  U(z)=&\frac{z^3}{1-z}+\frac{S_X(z^2)}{2}+S_X(z)\\
  &-\frac{z}{1-z}\cdot\Big(\frac{z}{1-z}+S_X(z)\Big)-\frac{S_X(z)^2}{2}.
\end{align*}
This is of the form $U(z)=G(z,S_X(z))$, where we define 
\begin{align*}
  G(z,y):=&\frac{z^3}{1-z}+\frac{S_X(z^2)}{2}+y\\
  &-\frac{z}{1-z}\cdot\Big(\frac{z}{1-z}+y\Big)-\frac{y^2}{2}.
\end{align*}
Note that $G(z,y)$ is analytic at $(\rho,\tau)$ (because $\rho<1$ and
$S_X(z^2)$ is analytic at $\rho$). Hence $U(z)$ has a singular expansion
at $\rho$ of the form
\begin{align*}
  U(z)=\tau'-c'\cdot Z+d'\cdot Z^2+e'\cdot Z^3+O(Z^4),
\end{align*}
and moreover, around $\rho$ we have
\begin{align*}
  U(z)&=G(z,S_X(z))\\
      &=G(\rho+(z-\rho),\tau-c\cdot Z+O(Z^2))\\
      &=G(\rho,\tau)+G_y(\rho,\tau)\cdot(-c\ \!Z)+O(Z^2),
\end{align*}
hence we have $\tau'=G(\rho,\tau)$ and more importantly 
$c'=c\cdot G_y(\rho,\tau)$. We have 
$$
G_y(z,y)=1-y-\frac{z}{1-z},
$$
and we are going to show that this cancels out at $(\rho,\tau)$
(so that $c'=0$). Recall that at $(z,y)=(\rho,\tau)$ the equations
$y=F(z,y)$ and $1=F_y(z,y)$ are satisfied. These equations read
\begin{align*}
  y&=\frac{z}{1-z}\cdot\Big( \exp\big(y+B(z)\big)-1\Big),\\
  1&=\frac{z}{1-z}\cdot\exp\big(y+B(z)\big).
\end{align*}
Subtracting the second equation from the first equation we get
$y-1=-L(z)$, which is also $0=1-y-L(z)=G_y(z,y)$. Since this equation is
satisfied at $(\rho,\tau)$ we conclude that $G_y(\rho,\tau)=0$, so that
$c'=0$. It remains to check that the leading singular term of $U(z)$ is
indeed $Z^3$ (i.e., we want to make sure that $e'\neq 0$). If it was not
the case we would have $U(z)=\tau+d'\cdot Z^2+O(Z^4)$, and using transfer
theorems it would imply that $[z^n]U(z)=o(\rho^{-n}n^{-5/2})$.

Note that an unrooted 3LP graph $\gamma$ with $n$ vertices gives rise to
not more than $n$ objects in $\cS_X$ (precisely it gives rise to
$n(\gamma)$ objects in $\cS_X$, where $n(\gamma)$ is the number of
dissimilar vertices of $\gamma$ that are adjacent to a star-leaf in the
split-decomposition tree), hence $n\cdot [z^n]U(z)\geq [z^{n-1}]S_X(z)$.
Since $[z^{n-1}]S_X(z)$ is $\Theta(\rho^{-n}n^{-3/2})$ we conclude that
$[z^n]U(z)=\Omega(\rho^{-n}n^{-5/2})$, and thus $e'\neq 0$. Using transfer
theorems we conclude that
$[z^n]U(z)\sim\frac{3\ \!e'}{4\sqrt{\pi}}\rho^{-n}n^{-5/2}$.


\subsection{Distance-hereditary graphs.}

Again we start the study with the rooted case (the study is very similar,
so that we give less details here). As stated in the remark after
Theorem~\ref{thm:dh-rooted-grammar}, $\cK$ and $\cS_C$ play symmetric
roles; hence have the same generating function, which we denote by $K(z)$.
We denote by $S(z)$ the generating function of $\cS_X$. Then from
Eq.~\eqref{eq:dhrs-s}, we obtain
\begin{align*}
  S(z)=\frac{(z+2K(z))^2}{1-z-2K(z)}.
\end{align*}
Now if we define $A(z)=z+K(z)+S(z)=z+K(z)+\frac{(z+2K(z))^2}{1-z-2K(z)}$
then Eq.~\eqref{eq:dhrs-k} yields
\begin{align*}
  K(z)=\exp\Big(\sum_{i\geq 1}\frac1{i}A(z^i)\Big)-1-A(z),
\end{align*}
so that $y=K(z)$ satisfies the functional equation
\begin{align*}
  y=\exp_{\geq 2}\Big(z+y+\frac{(z+2y)^2}{1-z-2y}\Big)\exp(B(z)),
\end{align*}
with the notation $\exp_{\geq d}(t)=\sum_{i\geq d}\frac{t^i}{i!}$, and 
with $B(z):=\sum_{i\geq 2}\frac1{i}A(z^i)$. 

As in Section~\ref{sec:3LP_asymp}, this is an equation of the form
$y=F(z,y)$, with $F(z,y)$ a power series with nonnegative coefficients and
with non-linear dependency on $y$. Thus, if we denote by $\rho$ the radius
of convergence of $K(z)$, then $K(z)$ converges to a finite positive
constant (denoted by $\tau$) when $z$ tends to $\rho$ from below; and
since $y=K(z)$ does not diverge when $z$ tends to $\rho$, then we must
have $1-z-2y>0$ at $(z,y)=(\rho,\tau)$ (no cancellation of the denominator
appearing inside the exponential). Moreover, since the number of
distance-hereditary graphs grows exponentially with $n$, we must have
$\rho<1$, from which we easily deduce that $B(z)$ is analytic at $\rho$,
and that $F(z,y)$ is analytic at $(\rho,\tau)$. Hence, as in
Section~\ref{sec:3LP_asymp}, the Drmota-Lalley-Woods theorem ensures that
$K(z)$ has a singular expansion of the form
$$
K(z)=\tau-c\cdot Z+O(Z^2),\ \ \mathrm{with}\ Z=\sqrt{1-z/\rho},
$$
and has an analytic continuation in a $\Delta$-domain. Hence we can apply
transfer theorems to obtain the asymptotic estimate
$[z^n]K(z)\sim\frac{c}{\sqrt{\pi}}\rho^{-n}n^{-3/2}$. Again we can use an
iterated scheme to evaluate the constants with increasing precision, we
obtain $\rho\approx 7.249751250\ldots$.

For the unrooted case, similarly as for 3LP graphs, we express the
generating function $U(z)$ of unrooted DH graphs in terms of $K(z)$, and
verify that the leading singular term is of order $Z^3$. Again we have to
use the fact that $K(z)$ admits a singular expansion up to terms of order
$Z^3$, of the form
$$
K(z)=\tau-c\ \!Z+d\ \!Z^2+e\ \!Z^3+O(Z^4).
$$
Eq.~\eqref{eq:dhu-main} of Theorem~\ref{thm:dh-unrooted-grammar} yields
$$
U(z)=T_K(z)+T_S(z)+T_{S-S}(z)-T_{K-S}(z)-T_{S\to S}(z).
$$ 
To express $T_K(z)+T_{S-S}(z)$ in terms of $K(z)$ we observe that
\begin{align*}
\cK&=\Set[\geq 2](\cZ+\cK+\cS)\\
&=\Set[\geq 3](\cZ+\cK+\cS)+\Set[2](\cZ+\cK+\cS)\\
&=\cT_K+\Set[2](\cK)+\Set_2(\cS)+\Set[2](\cZ)\\
&\quad+\cZ\times\cK+\cZ\times\cS+\cK\times\cS\\
&=\cT_K+\cT_{S-S}+\Set[2](\cZ)+\cZ\times\cK+\cZ\times\cS+\cK\times\cS.
\end{align*}
Hence
\begin{align*}
T_K(z)+T_{S-S}(z)&=K(z)-z\ \!K(z)-z\ \!S(z)-\ \!K(z)S(z).
\end{align*}
Next, we have
$$
T_{K-S}(z)=K(z)\ \!S(z),
$$
and 
$$
T_{S\to S}(z)=K(z)^2+S(z)^2.
$$
Finally, using $S(z)=\frac{(z+2K(z))^2}{1-z-2K(z)}$, 
 we find $U(z)=G(z,K(z))$, where
$$
G(z,y):=y-z^2-\frac{(z+2y)^3}{(1-z-2y)^2}.
$$
Remarkably, $U(z)$ admits here 
a rational expression in terms of $z$ and
$K(z)$, which was not the case for 3LP graphs (recall that
the expression of $U(z)$ involved a term $S_X(z^2)$). 

Similarly as for 3LP graphs, we note that $G(z,y)$ is analytic
at $(\rho,\tau)$, so that $U(z)$ admits a singular expansion of
the form
$$
U(z)=\tau-c'\ \!Z+d'\ \!Z^2+e'\ \!Z^3+O(Z^4),
$$  
 with the relation $c'=c\ \!G_y(\rho,\tau)$. We have
$$
G_y(z,y)=\frac{(1+z+2y)(4y^2+4zy+z^2-8y-4z+1)}{(1-z-2y)^3}.
$$
In order to verify that this cancels out at $(\rho,\tau)$,  
we again use the fact that at $(\rho,\tau)$, both equations $y=F(z,y)$
and $1=F_y(z,y)$ are satisfied. Defining $R(z,y)=z+y+\frac{(z+2y)^2}{1-z-2y}$, these equations read
\begin{align*}
y&=\exp(R(z,y)+B(z))-1-R(z,y),\\
1&=R_y(z,y)\exp(R(z,y)+B(z))-R_y(z,y).
\end{align*} 
Multiplying the first one by $R_y(z,y)$ and then subtracting the  second one (so as to eliminate $\exp(R(z,y)+B(z))$), we obtain the following equation, which is satisfied at $(\rho,\tau)$:
$$
0=\frac{4y^2+4zy+z^2-8y-4z+1}{(1-z-2y)^3}.
$$
We recognize the numerator as a factor in the numerator of $G_y(z,y)$,
from which we conclude that $G_y(\rho,\tau)=0$, and thus $c'=0$.
Similarly as for 3LP graphs, the fact that $[z^n]K(z)=\Theta(\rho^{-n}n^{-3/2})$ and $[z^n]U(z)\geq \frac{1}{n}[z^{n-1}]K(z)$ ensures
that $e'\neq 0$, and $[z^n]U(z)\sim \frac{3\ \!e'}{4\sqrt{\pi}}\rho^{-n}n^{-5/2}$. 


\section{Exhaustive Enumeration\label{sec:exhaustive}}

Since most of the classes enumerated in this paper, in their various
flavors (labeled/unlabeled, rooted/unrooted, connected/disconnected), had
no known enumeration, it became useful to have some reference enumerations
to confirm the correctness of the grammars we deduced.

To this end, we have used the vertex incremental characterization of the
studied classes of graphs. These are surprisingly readily available in the
graph literature, and provide a convenient---and thankfully rather
foolproof---way of finding reliable enumeration and exhaustive generation
of these classes of graphs.

\section{Conclusion}


In this paper, we have taken well-known characterization results by
established graph researchers~\cite{GiPa12}, and have turned these
characterizations into grammars, enumerations and asymptotics---for two
classes of graphs for which these were previously unknown.

This illustrates that a tool long known by graph theorists is a very
fruitful line of research in analytic combinatorics, of which this paper
is likely only the beginning.

Future questions in this same line may focus, for instance, on the
parameter analysis. For instance, Iriza~\cite[\S 7]{Iriza15} has already
empirically noted, that in the split-decomposition tree of an unrooted,
unlabeled distance-hereditary graph, the number of clique-nodes grows
approximately as $\sim 0.221n$ and the number of star nodes grows
approximately as $\sim 0.593n$. This offers some intuition as to what is a
typical ``shape'' for a distance-hereditary graph: many nodes concentrated
in a small number of cliques and then long filaments in between as in
Figure~\ref{fig:random-dh-52}. But a more qualitative investigation is
required.

Iriza also brings to light an issue with our methodology. While the
dissymmetry theorem solves many issues that have frustrated many
combinatoricians (the symmetries when enumerating unrooted trees), it does
provide a symbolic grammar for the unrooted graph classes. This prevents
us from efficiently randomly generating graphs~\cite[\S 3.2]{Iriza15}. An
interesting line of inquiry would be to refine the application of
cycle-pointing so that it is as straightforward as that of the dissymmetry
theorem.

Another promising avenue, is to investigate whether more complicated
classes of graphs can easily be enumerated. Any superset of the
distance-hereditary graphs (which are the totally decomposable graphs for
the split-decomposition) will necessarily involve the presence of
\emph{prime nodes} (internal graph labels which are neither star graphs
nor clique graphs). For instance, Shi~\cite{Shi15} has done an
experimental study of parity graphs (which have bipartite graphs as prime
nodes).

\section*{Acknowledgments}

This work was begun in 2008, when the second author was visiting the first
at Simon Fraser University; it was pursued in 2013-2014 during the third
author's post-doc at the same university; the results in this paper were
presented~\cite{ChFuLu14} at the Seventh International Conference on Graph
Transformation (ICGT 2014).

Between then and now, it has improved from the careful remarks and the
work of several of our students: Alex Iriza~\cite{Iriza15}, who provided
many of the figures, Jessica Shi, and Maryam Bahrani.

All of the figures (with the exception of Figure~\ref{fig:random-dh-52})
in this article were created in OmniGraffle~6~Pro.


\bibliographystyle{plain}
\bibliography{article2,homebrew}{}

\appendix


\section{Distance-Hereditary Grammar Simplification\label{app:dh-simplification}}

\begin{figure*}[b!]
  \centering
  \begin{minipage}[b]{.45\linewidth}
    \centering
    \includegraphics[scale=0.8]{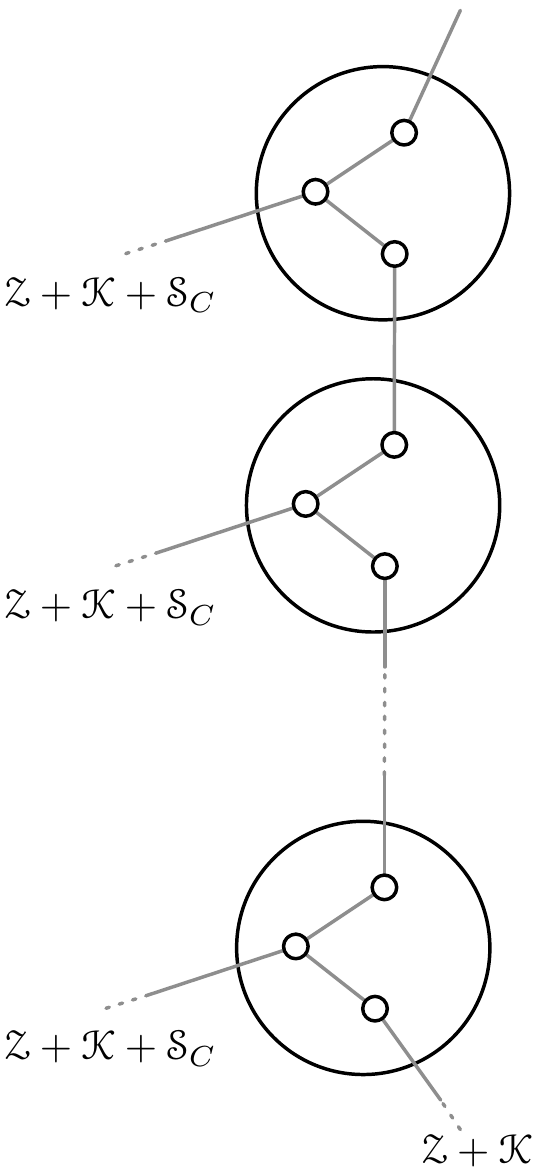}
    \subcaption{This case occurs when the last star in the subsequence of
      adjacent stars has \emph{only one} extremity.
      \label{fig:intuition-dh-grammar-simplification-a}}
  \end{minipage}\hspace{0.08\linewidth}%
  \begin{minipage}[b]{.45\linewidth}
    \centering
    \includegraphics[scale=0.8]{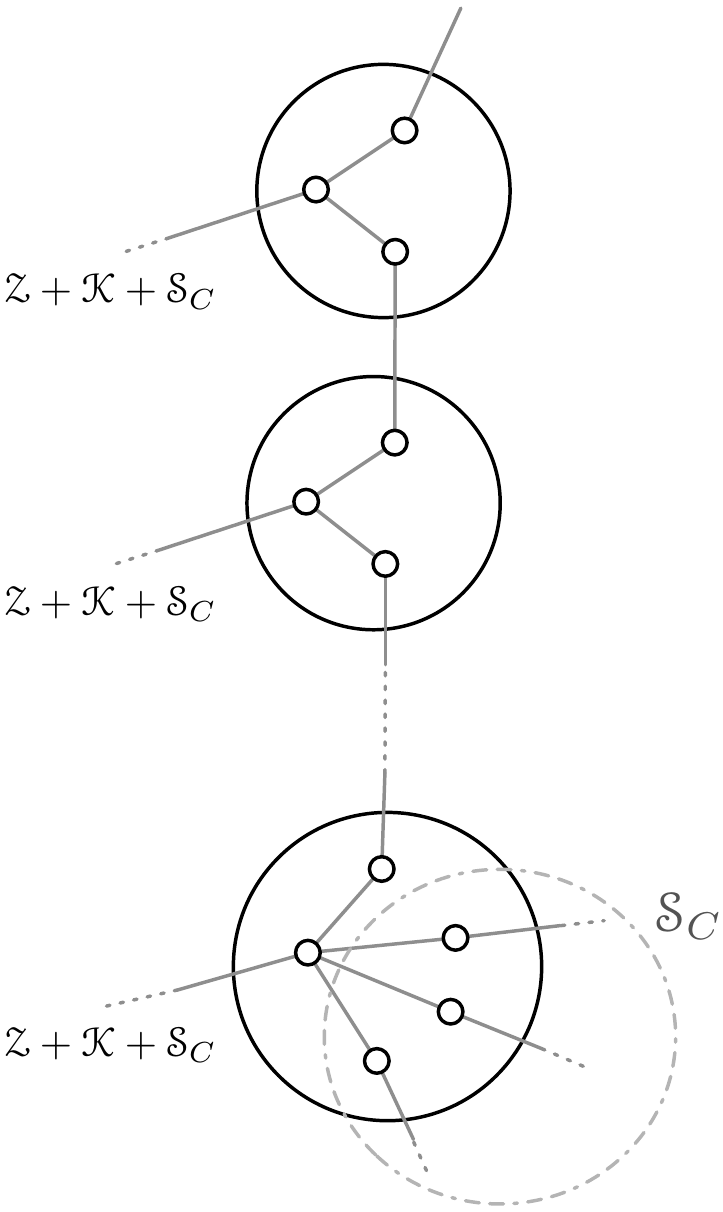}
    \subcaption{This case occurs when the last star in the subsequence of
      adjacent stars has \emph{at least two} extremities (here, it has
      three).\label{fig:intuition-dh-grammar-simplification-b}}
  \end{minipage}
  \caption{Combinatorial intuition behind the derivation of
    Appendix~\ref{app:dh-simplification}.\label{fig:intuition-dh-grammar-simplification}}
\end{figure*}

The class $\clsDHrl$ of distance-hereditary graphs rooted at a vertex is
originally specified by
\begin{align*}
  \clsDHrl    &= \clsAtom_{\mLeaf}\times\left(\cls{K} + \cls[C]{S} + \cls[X]{S}\right)\\
  \cls{K}     &= \Set[\geqslant 2]{\clsAtom + \cls[C]{S} + \cls[X]{S}}\\
  \cls[C]{S}  &= \Set[\geqslant 2]{\clsAtom + \cls{K} + \cls[X]{S}}\\
  \cls[X]{S}  &= \left(\clsAtom +\cls{K}+\cls[C]{S}\right)\times
                      \Set[\geqslant 1]{\clsAtom + \cls{K} + \cls[X]{S}}\text{.}
\end{align*}
The point of this appendix is to prove that the last equation can be
simplified to
\begin{align*}
  \cls[X]{S}  &= \Seq[\geqslant 2]{\clsAtom + \cls{K} + \cls[C]{S}}\text{.}
\end{align*}
Although we first provide a straightforward formal derivation, we then
follow it up with an intuitive explanation.

\begin{proof}
  Indeed, while the elements of a \Set have symmetries that are hard to
  take into account, this is not the case for sets of size 1, therefore
  \begin{align*}
    \Set[\geqslant 1]{\cls{U}} = \cls{U} + \Set[\geqslant 2]{\cls{U}}\text{.}
  \end{align*}
  By combining this fact with the definition of \cls[C]{S},
  \begin{align*}
    \cls[C]{S}  = \Set[\geqslant 2]{\clsAtom + \cls{K} + \cls[X]{S}}\text{,}
  \end{align*}
  we have that (parentheses in the right hand side purely for emphasis)
  \begin{align*}
    \Set[\geqslant 1]{\clsAtom + \cls{K} + \cls[X]{S}} \equiv
    \left(\clsAtom + \cls{K} + \cls[X]{S}\right) + \cls[C]{S}
  \end{align*}
  hence,
  \begin{align*}
    \cls[X]{S} &= \left(\clsAtom +\cls{K}+\cls[C]{S}\right)\times
                   \left(\clsAtom + \cls{K} + \cls[X]{S} + \cls[C]{S}\right)\\
               &= \cls[X]{S}\times\left(\clsAtom +\cls{K}+\cls[C]{S}\right)+
                  \left(\clsAtom +\cls{K}+\cls[C]{S}\right)^2
  \end{align*}
  we then proceed to manipulate this specification purely symbolically,
  implying
  \begin{align*}
    \cls[X]{S} \left[1 - \left(\clsAtom + \cls{K} + \cls[C]{S}\right)\right] =
    \left(\clsAtom + \cls{K} + \cls[C]{S}\right)^2
  \end{align*}
  and thus
  \begin{align*}
  \cls[X]{S} &=
    \frac{\left(\clsAtom + \cls{K} + \cls[C]{S}\right)^2}
         {1 - \left(\clsAtom + \cls{K} + \cls[C]{S}\right)}\\
             &= \left(\clsAtom + \cls{K} + \cls[C]{S}\right)^2\times
               \Seq{\clsAtom + \cls{K} + \cls[C]{S}}\\
             &=\Seq[\geqslant 2]{\clsAtom + \cls{K} + \cls[C]{S}}\text{.}
  \end{align*}
  Finally
  \begin{align*}
    \cls[X]{S} = \Seq[\geqslant 2]{\clsAtom + \cls{K} + \cls[C]{S}}\text{.}
  \end{align*}
\end{proof}

\begin{remark}
  To understand this simplification from a combinatorial perspective,
  imagine that we have a connected subsequence of star-nodes connected by
  their extremities.
  
  Without loss of generality, we can assume that all \emph{but the last}
  of these internal star-nodes have only two extremities\footnote{The
    first star-node of the subsequence to have more than one extremity is
    the ``last'' star-node of that particular subsequence. In particular,
    it is possible for the subsequence to only have one single
    star-node.}---the one through which they are entered, and another one.
  We are then either in the situation illustrated by
  Figure~\ref{fig:intuition-dh-grammar-simplification-a} (in which the
  last star-node of the subsequence only has one additional extremity) or
  by Figure~\ref{fig:intuition-dh-grammar-simplification-b} (in which the
  last star-node has several extremities).
  
  This subsequence of adjacent star-nodes connected by their extremities,
  translates to the grammar by a recursive expansion of the
  $\cramped{\cls[X]{S}}$ rule: each of these has a
  $\cramped{(\clsAtom + \cls{K} + \cls[C]{S})}$ child for the center of
  the star, and then one other children for the other extremity. This is
  repeated until we have reached the last adjacent star-node in the
  subsequence which can either have one or multiple extremities:
  \begin{itemize}
  \item If it has only one extremity, then this extremity connects to
    either a leaf or to a clique, thus $\clsAtom + \cls{K}$
    (Figure~\ref{fig:intuition-dh-grammar-simplification-a}).
  \item Otherwise, it has two or more undistinguished extremities, in
    which case we can \emph{pretend} that this set of extremities is a
    $\cramped{\cls[C]{S}}$ term
    (Figure~\ref{fig:intuition-dh-grammar-simplification-b}).
  \end{itemize}

  \noindent Recall that the original interpretation of
  $\cramped{\cls[X]{S}}$,
  \begin{align*}
    \cls[X]{S}  &= \left(\clsAtom +\cls{K}+\cls[C]{S}\right)\times
                  \Set[\geqslant 1]{\clsAtom + \cls{K} + \cls[X]{S}}\text{,}
  \end{align*}
  is as follows: a distinguished center which can lead to either a leaf, a
  clique, or a star-node entered through its center; and a set of
  undistinguished extremities, each of which can lead to either a leaf, a
  clique, or another star-node entered through an extremity.

  The new interpretation follows the figures: we have a sequence of
  $\cramped{(\clsAtom + \cls{K} + \cls[C]{S})}$ terms for the center of
  each of the adjacent star-nodes (and we have at least one such
  star-node), and finally another such term to cover both possibilities,
  where the final star-node either has one extremity or several. This is
  equivalent to having a sequence of at least two of these terms, hence
  the simplified equation.
\end{remark}

\begin{figure*}
  \centering
  \includegraphics[scale=0.6]{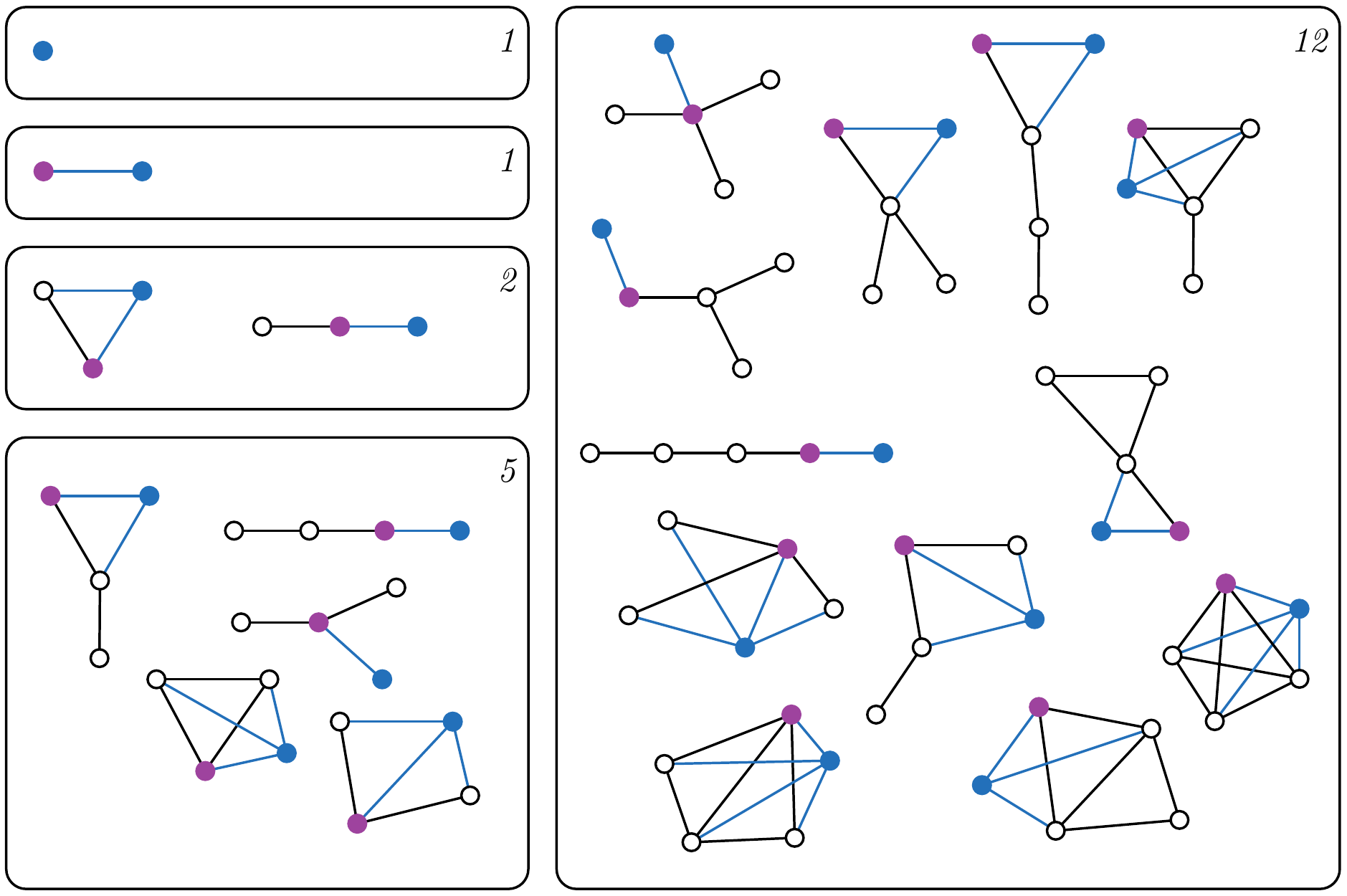}
  \caption{\label{fig:3lp-enum}All unrooted, unlabeled 3-leaf power graphs
    of sizes 1 through 5, beginning the enumeration: 1, 1, 2, 5, 12, ....
    The coloring of the vertices illustrate one possible way to derive the
    graphs through vertex incremental operations sketched in
    Section~\ref{sec:exhaustive}: the newly added vertex is in blue, while
    the existing vertex it is added in reference to is in purple.}
\end{figure*}

\end{document}